 \documentclass[final,3p,times]{elsarticle}

\usepackage{amssymb}
\usepackage{amsmath, amsthm, amsfonts,ifthen}
\usepackage[english]{babel}
\usepackage{color}
\usepackage{dsfont}
\newtheorem{theorem}{Theorem}

\newtheorem{definition}[theorem]{Definition}

\newtheorem{lemma}[theorem]{Lemma}

\newtheorem{proposition}[theorem]{Proposition}
\newtheorem{remark}[theorem]{Remark}

\numberwithin{equation}{section}

\begin{document}

\begin{frontmatter}



\title{Penalization method for reflected BDSDEs with two-sided jumps and driven by L\'{e}vy process}

 \author{Mohamed Marzougue}
 \address{Laboratory of Analysis and Applied Mathematics (LAMA), faculty of sciences Agadir, Ibn Zohr University,
Agadir, Morocco.\\mohamed.marzougue@edu.uiz.ac.ma}

\begin{abstract}
In this paper, we prove the existence and uniqueness of the solution to reflected backward doubly stochastic differential equations driven by Teugels martingales associated with a L\'{e}vy process where the barrier process is not necessarily right continuous by approximating such equations by a new version of penalization method.
\end{abstract}

\begin{keyword}
Reflected backward doubly stochastic differential equations \sep L\'{e}vy process \sep irregular barrier \sep penalization method

60G20 \sep 60H05 \sep 60H15
\end{keyword}

\end{frontmatter}


\section{Introduction}

After Pardoux and Peng \cite{PP1} have introduced the notion of non-linear backward stochastic differential equations (BSDEs in short) in a Brownian setting, the authors introduced a new kind of BSDEs in \cite{PP2}, that is a class of backward doubly stochastic differential equations (BDSDEs in short) with two different directions of stochastic integrals with respect to two independent Brownian motions. Precisely, these equations take the form
\begin{equation}\label{bdsde}
Y_{t}=\xi +\int_{t}^{T} f(s,Y_{s},Z_{s})ds+\int_{t}^{T} g(s,Y_{s},Z_{s})dB_s-\int_{t}^{T}Z_{s}dW_{s},\quad 0\leq t\leq T.
\end{equation}
where the integral driven by $dW_s$ is the standard forward stochastic It\^{o}'s integral and the integral driven by $dB_s$ is the backward stochastic Kunita-It\^{o}'s integral. The authors established the existence and uniqueness of solution to BDSDE (\ref{bdsde}) under the square integrability on the data $(\xi,f,g)$ with Lipschitzian drivers $f$ and $g$. As application, the authors gave a probabilistic representation for a class of quasi-linear stochastic partial differential equations. Following this, Ren et al. \cite{RLH} have considered BDSDEs driven by Teugels martingales associated with a L\'{e}vy process satisfying some moment condition and an independent Brownian motion. Later, there are several works which have focused on developing the theory of BDSDEs in different direction (see for instance \cite{Aman3,fa1,fa2,owo1,owo2,owo3,sow,sagna}).

As a variation of BDSDEs, Bahlali et al. \cite{bahlali} were introduced reflected BDSDEs (RBDSDEs in short) where an additional nondecreasing process $K$ is added to the standard BDSDEs (\ref{bdsde}) in order to keep the $Y$-component of the solution above a certain lower continuous process, called barrier (or obstacle), and to do so in a minimal fashion. The authors have studied the case when the coefficient $f$ is continuous, and they proved the existence of a minimal and maximal solutions in a Brownian setting. Following this, Ren \cite{Ren} has considered RBDSDEs driven by Teugels martingales associated with a L\'{e}vy process, in which the barrier process is right continuous with left limits. We stress that the theory of RBDSDEs has been extended to the case where the barrier is not necessarily continuous and/or a larger filtration than the Brownian filtration, or by weakening the assumptions on the coefficients, by several authors, we quote \cite{Aman1,Aman2,Aman4,Hu,LL,Ren} and references therein. In all of the mentioned works, the barrier has been assumed to be at least right continuous. Recently, Berrhazi et al. \cite{berhazi} considered RBDSDEs when the barrier is not necessarily right-continuous by inspiring on the work of Grigorova et al. \cite{GIOOQ} which is the first one dealing with right upper semi-continuous barrier reflected BSDEs. For more developments on Reflected BSDEs when the barrier is not necessarily right-continuous, we refer to \cite{Akdim,BO,GIOQ,KRS,Ma1,ME:2019,ME:2020}. More recently, Marzougue and Sagna \cite{MS} extended the work of Berrhazi et al. \cite{berhazi} to the case when the noise is driven also by an independent Poisson random measure under the so-called stochastic Lipschitz condition on the drivers.

The penalization method is one of useful tools to establish the existence result for RBSDEs (e.g., \cite{Ess,HO,LX}) and RBDSDEs (e.g., \cite{Aman4}). Very recently, Marzougue \cite{Ma2} has proved a new monotonic limit theorem for regulated processes (in which the trajectories have just left and right limits), from that proved by Peng \cite{Peng}. As application, the author has established the existence theorem of RBSDEs when the barrier has regulated trajectories (not necessarily right-continuous) by means of penalization method.

Motivated by the Marzougue's monotonic limit theorem, we can now approximate the solution of RBDSDEs with regulated barrier by a modified penalized schema. To this end, we mainly consider, in this paper, the problem of RBDSDEs driven by Teugels martingales associated with a L\'{e}vy process (RBDSDEL in short), and we prove the existence and uniqueness of the solution of such equations by means of penalization method.

The paper is organized as follows: In section \ref{s1} we give some notations and preliminaries, and we formulate our problem. Section \ref{s2} is devoted to establish the well-posedness of our problem which corresponds to a class of RBDSDEL where the barrier process is not necessarily right-continuous by means of penalization method.

\section{Background}\label{s1}
\subsection{Preliminaries and notations}

Let $T$ strictly positive real number and let $(\Omega,\mathcal{F},\mathbb{P},(\mathcal{F}_{t})_{t\leq T},(B_{t})_{t\leq T},(\mathfrak{L}_{t})_{t\leq T})$ be a complete Brownian-L\'{e}vy space in $\mathbb{R}\times\mathbb{R}^\ast$, with L\'{e}vy measure $\nu$, i.e. $(\Omega,\mathcal{F},\mathbb{P})$ is a complete probability space, $(B_{t})_{t\leq T}$ is a standard Brownian motion in $\mathbb{R}$ and $(\mathfrak{L}_{t})_{t\leq T}$ is a $\mathbb{R}^\ast$-valued pure-jump L\'{e}vy process independent of $(B_{t})_{t\leq T}$, which corresponds to a standard L\'{e}vy measure $\nu$ satisfying the following conditions:
\begin{itemize}
\item[$(i)$] $\int_\mathbb{R}(1\wedge x^2)\nu(dx)<+\infty,$
\item[$(ii)$] $\int_{(-\varepsilon ,\varepsilon )^{c}}e^{\lambda |x|}\nu(dx)<+\infty$ for each $\varepsilon >0$ and some $\lambda>0$.
\end{itemize}
For each $t\leq T$, we define
$$\mathcal{F}_t\triangleq\mathcal{F}_{t,T}^B\vee\mathcal{F}_{0,t}^L,$$
where for any process $(\eta_t)_{t\leq T}$; $\mathcal{F}_{s,t}^\eta= \sigma\{\eta_r-\eta_s;\; s\leq r\leq t\}\vee\mathcal{N}$ and $\mathcal{N}$ denotes the class of $\mathbb{P}$-null sets of $\mathcal{F}$.
Note that the collection $(\mathcal{F}_{t})_{t\leq T}$ is neither increasing nor decreasing, so it does not constitute a filtration. However, $(\mathcal{G}_{t})_{t\leq T}$ defined as $\mathcal{G}_t\triangleq\mathcal{F}_{T}^B\vee\mathcal{F}_{t}^L$ is a filtration which contains $(\mathcal{F}_{t})_{t\leq T}$.

We denote by $\mathfrak{L}_{t-}=\lim_{s\nearrow t}\mathfrak{L}_{s}$ and $\Delta \mathfrak{L}_{t}=\mathfrak{L}_{t}-\mathfrak{L}_{t-}$. The power-jumps of the L\'{e}vy process $\mathfrak{L}$ are defined by $\mathfrak{L}_{t}^{(1)}=\mathfrak{L}_{t}\quad \mbox{and}\quad \mathfrak{L}_{t}^{(i)}=\sum_{0< s\leq t}(\Delta \mathfrak{L}_{s})^{i},\ i\geq 2$. Define $Y_{t}^{(i)}=\mathfrak{L}_{t}^{(i)}-\mathbb{E}[\mathfrak{L}_{t}^{(i)}]=\mathfrak{L}_{t}^{(i)}-t\mathbb{E}[\mathfrak{L}_{1}^{(i)}]$ for all $i\geq1$, the so-called Teugels martingales. We associate with the L\'{e}vy process $\mathfrak{L}$ the family of processes $(H^{(i)})_{i\geq1}$ defined by $H_{t}^{(i)}=\sum_{j=1}^{i}\alpha_{ij}Y_{t}^{(j)}$.
The martingales $H^{(i)}$, called the orthonormalized $i$th-power-jump processes, are strongly orthogonal and its predictable quadratic variation process
is $\langle H^{(i)},H^{(j)}\rangle_t=\delta_{ij}t$. For more details on Teugels martingales, one can see Bertoin \cite{Ber} and also Nualart and Schoutens \cite{NS:2000}.

We will denote by
\begin{itemize}
  \item $|.|$ the Euclidian norm on $\mathbb{R}^{d}$,
  \item $\mathcal{T}_{[t,T]}$ the set of all stopping times $\tau$ with values in $[t,T]$,
  \item $\mathcal{P}$ (resp. $\mathcal{O}(\mathbb{R}^{d})$) the predictable (resp. Optional) $\sigma$-algebra on $\Omega\times [0,T]$ (resp. on $\mathbb{R}^{d}$).
\end{itemize}

Let us introduce some spaces:
\begin{itemize}
  \item $\mathcal{S}^{2}$ is the space of  $\mathbb{R}$-valued and optional processes $(Y_t)_{t\leq T}$ such that
       $$\|Y\|_{\mathcal{S}^{2}}^{2}=\mathbb{E}\left[\operatorname*{ess
\,sup}_{\tau\in\mathcal{T}_{[0,T]}}|Y_{\tau}|^{2}\right] < +\infty.$$
  \item $\mathcal{H}^{2}$ is the space of $\mathbb{R}^d$-valued and predictable processes $(Z_t)_{t\leq T}$ such that
  $$\|Z\|_{\mathcal{H}^{2}}^{2}=\mathbb{E}\left[\int_0^T|Z_t|^{2}dt\right] < +\infty.$$
\item $\ell^2=\left\{z=(z_t)_{t\geq0};\ \|z\|_{\ell^2}=\left(\sum_{k=1}^\infty|z_t^{(k)}|^2\right)^\frac{1}{2}<+\infty\right\}$.
\item $\mathcal{H}^{2}(\ell^2)$: the space of the $\mathcal{F}_t$--predictable and $\ell^2$-valued processes $(Z_{t})_{t\leq T}$ such that
$$\|Z\|^2_{\mathcal{H}^{2}(\ell^2)}:=\mathbb{E}\left[\int_{0}^{T}\|Z_{t}\|_{\ell^2}^2dt\right]
=\sum_{k=1}^{\infty}\mathbb{E}\left[\int_{0}^{T}|Z_{t}^{(k)}|^2dt\right]<+\infty.$$
\item $\mathcal{B}^2:=\mathcal{S}^{2}\times\mathcal{H}^{2}(\ell^2)$ is a Banach space endowed with the norm
$$\|(Y,Z)\|^2_{\mathcal{B}^{2}}=\|Y\|_{\mathcal{S}^{2}}^{2}+\|Z\|^2_{\mathcal{H}^{2}(\ell^2)}.$$
\end{itemize}

\begin{remark}
Let $\beta>0$. There exists an equivalent norm $\|(.,.)\|_{\mathcal{B}^{2}_\beta}$ to the norm $\|(.,.)\|_{\mathcal{B}^{2}}$, defined in the Banach space $\mathcal{B}^2_\beta:=\mathcal{S}^{2}_\beta\times\mathcal{H}^{2}_\beta(\ell^2)$ as
$$\|(Y,Z)\|^2_{\mathcal{B}^{2}_\beta}=\|Y\|_{\mathcal{S}^{2}_\beta}^{2}+\|Z\|^2_{\mathcal{H}^{2}_\beta(\ell^2)}=\mathbb{E}\left[\operatorname*{ess
\,sup}_{\tau\in\mathcal{T}_{[0,T]}}e^{\beta \tau}|Y_{\tau}|^{2}\right]+\mathbb{E}\left[\int_{0}^{T}e^{\beta t}\|Z_{t}\|_{\ell^2}^2dt\right].$$
\end{remark}
\begin{definition}
A pair of functions $(f,g)$ is said to be a pair of Lipschitz drivers if
\begin{itemize}
\item  $f,g : \Omega \times [0,T] \times \mathbb{R} \times \ell^{2}\longrightarrow \mathbb{R}$ are progressively measurable.
\item $\mathbb{E}\int_0^T(|f(t,0,0)|^2+|g(t,0,0)|^2)dt<+\infty$.
\item There exists some nonnegative constants $L$ and $0<\alpha<\frac{1}{2}$ such that $\forall (t,y,z,y',z')\in [0,T]\times\mathbb{R} \times \ell^{2}\times \mathbb{R} \times \ell^{2}$
            $$|f(t,y,z)-f(t,y',z')|\leq L\left(|y-y'|+\|z-z'\|_{\ell^2}\right)$$
and
$$|g(t,y,z)-g(t,y',z')|^2\leq L|y-y'|^2+\alpha\|z-z'\|_{\ell^2}^2.$$
\end{itemize}
\end{definition}

\begin{definition}
We say that a function $\mathcal{Y}:[0,T]\rightarrow \mathbb{R}$ has regulated trajectories (or l\`{a}dl\`{a}g) if $\mathcal{Y}$ has a
left limit in each point of $]0,T]$, and a right limit in each point of $[0,T[$.
For a process $\mathcal{Y}$ with regulated trajectories, we denote
\begin{itemize}
  \item $\mathcal{Y}_{t-}=\lim\limits_{s\nearrow t}\mathcal{Y}_s$ the left-hand limit of $\mathcal{Y}$ at $t\in]0,T]$, $(\mathcal{Y}_{0-}=\mathcal{Y}_0)$, $\mathcal{Y}_-:=(\mathcal{Y}_{t-})_{t\leq T}$ and $\Delta \mathcal{Y}_t:=\mathcal{Y}_{t}-\mathcal{Y}_{t-}$ the size of the left jump of $\mathcal{Y}$ at $t$.
  \item $\mathcal{Y}_{t+}=\lim\limits_{s\searrow t}\mathcal{Y}_s$ the right-hand limit of $\mathcal{Y}$ at $t\in[0,T[$, $(\mathcal{Y}_{T+}=\mathcal{Y}_T)$, $\mathcal{Y}_+:=(\mathcal{Y}_{t+})_{t\leq T}$ and $\Delta_+\mathcal{Y}_t:=\mathcal{Y}_{t+}-\mathcal{Y}_{t}$ the size of the right jump of $\mathcal{Y}$ at $t$.
  \item For all $t\leq T$, $\mathcal{Y}_t=\mathcal{Y}^\ast_t+\sum_{s<t}\Delta_+\mathcal{Y}_s$ where $\mathcal{Y}^\ast$ is the right-continuous part of the process $\mathcal{Y}$ and $\sum_{s<t}\Delta_+\mathcal{Y}_s$ stands its purely jumping part
consisting of right jumps such that $\sum_{s<t}|\Delta_+\mathcal{Y}_s|<+\infty$ a.s.
\end{itemize}
\end{definition}

\subsection{Problem's formulation}
Let us now announce the definition of reflected BDSDEs driven by Teugels martingales associated with a pure-jump L\'{e}vy process. Let $\xi\in\mathcal{S}^2$. For all $t\leq T$, we define $\hat{\xi}_t:=\limsup\limits_{s\uparrow t,\;s<t}\xi_s$. $\hat{\xi}$ is predictable (see Theorem 90, page 225 in \cite{DM1}). It is left upper semi-continuous and is called the left upper semi-continuous envelope of $\xi$ (cf, Definition \ref{left} in Appendix).
\begin{definition}[RBDSDELs]
 Let $(f,g)$ be a pair of Lipschitz drivers and $\xi$ be an irregular barrier belongs to $\mathcal{S}^2$. The triple of processes $(Y,Z,K)$ is said to be solution to RBDSDEL associated with parameters $(f,g,\xi)$ if
\begin{eqnarray}
&&\hspace{-0.5cm}(i)\; (Y,Z,K)\in \mathcal{S}^{2}\times \mathcal{H}^{2}(\ell^2)\times\mathcal{S}^{2},\nonumber\\
&&\hspace{-0.5cm}(ii)\; Y_{t}=\xi_T +\int_{t}^{T} f(s,Y_{s},Z_{s})ds+\int_{t}^{T} g(s,Y_{s},Z_{s})dB_s+K_{T}-K_{t}-\sum_{k=1}^{\infty}\int_t^TZ_{s}^{(k)}dH_{s}^{(k)} \qquad t\leq T,\nonumber\\
&&\hspace{-0.5cm}(iii)\; Y_t\geq \xi_t \quad \forall t\leq T,\nonumber\\
&&\hspace{-0.5cm}(iv)\; K \mbox{ is a nondecreasing predictable process with regulated trajectories}\nonumber\\
 &&\quad \mbox{ such that }K_0=0,\;\; \mathbb{E}[K_T]<+\infty \mbox{ and }\quad\displaystyle\int_0^T(Y_{t-}-\hat\xi_{t})dK^{\ast}_t+\sum_{t<T}(Y_{t}-\xi_{t})\Delta_+K_t=0 \;a.s.\nonumber
\end{eqnarray}
\end{definition}

\begin{remark}
If $(Y,Z,K,C)$ is a solution to RBSDELs associated with parameters $(f,g,\xi)$, then the process $Y$ has regulated trajectories. Moreover, the process $\left(Y_t+\int_0^tf(s,Y_s,Z_s)ds\right)_{t\leq T}$ is an optional strong supermartingale.
\end{remark}

\begin{proposition}
Let $(Y,Z)\in\mathcal{S}^{2}\times\mathcal{H}^{2}(\ell^2)$ where $Y$ has a regulated trajectories. Then the process $\left(\sum\limits_{k=1}^{\infty}\int_0^{t} Y_{s-}Z_{s}^{(k)}dH_{s}^{(k)}\right)_{t\leq T}$ is a martingale.
\end{proposition}

\section{The main result: Existence and uniqueness of the solution}\label{s2}
\subsection{The uniqueness}

\begin{proposition}
Let $(f,g)$ be a pair of Lipschitz drivers and $\xi$ be a completely irregular barrier. The RBDSDEL associated with parameters $(f,g,\xi)$ has at most one solution.
\end{proposition}
\begin{proof}
Let us consider two solutions $(Y,Z,K,)$ and $(Y',Z',K')$ of RBDSDEL associated with parameters $(f,g,\xi)$. Denote $\overline{\Re}=\Re-\Re'$ for $\Re\in\{Y,Z,K\}$. By applying the Proposition \ref{pro} to $|\overline{Y}_t|^2$, we get
\begin{eqnarray*}
|\overline{Y}_t|^2+\int_t^T\|\overline{Z}_{s}\|_{\ell^2}^2ds
&=&2\int_t^T\overline{Y}_s(f(s,Y_{s},Z_{s})-f'(s,Y'_{s},Z'_{s}))ds
+2\int_t^T\overline{Y}_s(g(s,Y_{s},Z_{s})-g'(s,Y'_{s},Z'_{s}))dB_s\\
&&+2\int_t^T\overline{Y}_{s-}d\overline{K}^\ast_{s}-2\sum_{k=1}^{\infty}\int_t^T\overline{Y}_{s-}\overline{Z}_{s}^{(k)}dH_{s}^{(k)}
+\int_t^T|g(s,Y_{s},Z_{s})-g'(s,Y'_{s},Z'_{s})|^2ds\\
&&-\sum_{t< s\leq T}|\Delta_-\overline{Y}_{s}|^2-\sum_{t\leq s<T}|\Delta_+\overline{Y}_{s}|^2+2\sum_{t\leq s<T}\overline{Y}_{s}\Delta_+\overline{K}_{s}\\
&\leq&(3L+\varrho L^2)\int_t^T|\overline{Y}_s|^2ds+\left(\frac{1}{\varrho}+\alpha\right)\int_t^T\|\overline{Z}_s\|_{\ell^2}^2ds+2\int_t^T\overline{Y}_{s-}d\overline{K}^\ast_{s}
+2\sum_{t\leq s<T}\overline{Y}_{s}\Delta_+\overline{K}_{s}\\
&&+2\int_t^T\overline{Y}_s(g(s,Y_{s},Z_{s})-g'(s,Y'_{s},Z'_{s}))dB_s-2\sum_{k=1}^{\infty}\int_t^T\overline{Y}_{s-}\overline{Z}_{s}^{(k)}dH_{s}^{(k)}
\end{eqnarray*}
for some $\varrho>0$, where we have used the Lipschitz property of $(f,g)$. Thanks to the Skorokhod condition on $K$, we have
\begin{eqnarray*}
\int_t^T\overline{Y}_{s-}d\overline{K}^\ast_{s}&=&\int_t^T(Y_{s-}-\hat{\xi}_{s})dK^\ast_{s}-\int_t^T(Y'_{s-}-\hat{\xi}_{s})dK^\ast_{s}
+\int_t^T(Y'_{s-}-\hat{\xi}_{s})dK^{'\ast}_{s}
-\int_t^T(Y_{s-}-\hat{\xi}_{s})dK^{'\ast}_{s}\leq 0
\end{eqnarray*}
and
\begin{eqnarray*}
\sum_{t\leq s<T}\overline{Y}_{s}\Delta_+\overline{K}_{s}&=&\sum_{t\leq s<T}(Y_{s}-\xi_{s})\Delta_+K_{s}-\sum_{t\leq s<T}(Y'_{s}-\xi_{s})\Delta_+K_{s}
+\sum_{t\leq s<T}(Y'_{s}-\xi_{s})\Delta_+K'_{s}-\sum_{t\leq s<T}(Y_{s}-\xi_{s})\Delta_+K'_{s}\leq 0.
\end{eqnarray*}
Then for $\varrho>\frac{1}{1-\alpha}$ we obtain
\begin{equation*}
\mathbb{E}|\overline{Y}_t|^2+\mathbb{E}\int_t^T\|\overline{Z}_{s}\|_{\ell^2}^2ds\leq C_{_{L,\alpha}}\mathbb{E}\int_t^T|\overline{Y}_s|^2ds
\end{equation*}
where $C_{_{L,\alpha}}$ is a positive constant which depends on $L$ and $\alpha$. Consequently, according to Gronwall's lemma we obtain $Y=Y'$, $Z=Z'$ and thus $K=K'$.
\end{proof}

\subsection{The existence via penalization method}

Let us first consider the special case when the coefficient $g$ does not depend on the solution. We put $g(t,y,z):=\widetilde{g}(t)$ such that $\mathbb{E}\int_0^T|\widetilde{g}(t)|^2dt<+\infty$.
\begin{lemma}\label{tt}
Let $f$ be a Lipschitz driver and $\xi$ be a completely irregular barrier. Then the RBDSDEL associated with parameters $(f,\widetilde{g},\xi)$ admits a unique solution.
\end{lemma}
\begin{proof}
We rely on a penalization method by approximating the irregular barrier $\xi$, which has been introduced by Klimsiak et al. \cite{KRS} and later by Marzougue \cite{Ma2}. For each $n\in\mathbb{N}$, we consider the following penalized version of BDSDEL
\begin{eqnarray}\label{penal}
Y^n_{t}&=&\xi_T +\int_{t}^{T} f(s,Y^n_s,Z^n_s)ds+\int_{t}^{T} \widetilde{g}(s)dB_s-\sum_{k=1}^{\infty}\int_t^TZ_{s}^{(k),n}dH_{s}^{(k)}\nonumber\\
&&+n\int_t^T(Y^n_{s}-\xi_{s})^-ds+\sum_{t\leq \sigma_{n,i}<T}(Y^n_{\sigma_{n,i}+}-\xi_{\sigma_{n,i}})^-
\end{eqnarray}
with specially defied arrays of stopping times $\{\sigma_{n,i}\}$ exhausting right-side jumps of $\xi$.
We define $\{\sigma_{n,i}\}$ inductively. We fist set
 $$\displaystyle\left\{
  \begin{array}{ll}
    \sigma_{1,0}=0 , & \hbox{} \\
    \sigma_{1,i}=\inf\{t>\sigma_{1,i-1}\;\; |\;\;  \Delta_+\xi_t<-1 \}\wedge T, & \hbox{\quad$i=1,...,k_1$ for $k_1\in\mathbb{N}$.}
  \end{array}
\right.$$
Next, for $n\in\mathbb{N}$ and given array $\{\sigma_{n,i}\}$ we set
 $$\displaystyle\left\{
  \begin{array}{ll}
    \sigma_{n+1,0}=0 , & \hbox{} \\
    \sigma_{n+1,i}=\inf\{t>\sigma_{n+1,i-1}\;\; |\;\; \Delta_+\xi_t<-\frac{1}{n+1}\}\wedge T, & \hbox{\quad$i=1,...,j_{n+1}$}
  \end{array}
\right.$$
with $j_{n+1}$ is chosen so that $\mathbb{P}(\sigma_{n+1,j_{n+1}}<T)\rightarrow 0$ as $n\rightarrow +\infty$ and
$$\sigma_{n+1,i}=\sigma_{n+1,j_{n+1}}\vee\sigma_{n,i-j_{n+1}-1},\quad i=j_{n+1}+1,...,k_{n+1},\quad k_{n+1}=j_{n+1}+k_{n}+1.$$

According to the work of Ren et al. \cite{RLH}, on each interval $(\sigma_{n,i-1},\sigma_{n,i}]$, $i=1,...,k_{n}+1$ with $\sigma_{n,k_{n}+1}=T$, there exists a unique process $(Y^n,Z^n)$ solution of the following BDSDEL
\begin{eqnarray}\label{mod}
Y^n_{t}&=&\xi_{\sigma_{n,i}}\vee Y^n_{\sigma_{n,i}+}+\int_{t}^{\sigma_{n,i}} f(s,Y^n_s,Z^n_s)ds+\int_{t}^{\sigma_{n,i}} \widetilde{g}(s)dB_s\nonumber\\
&&+n\int_t^{\sigma_{n,i}}(Y^n_{s}-\xi_{s})^-ds-\sum_{k=1}^{\infty}\int_t^TZ_{s}^{(k),n}dH_{s}^{(k)}\qquad t\in(\sigma_{n,i-1},\sigma_{n,i}]
\end{eqnarray}
with the convention $Y^n_{T}=\xi_T$ and $Y^n_{0}=\xi_{0}\vee Y^n_{0+}$. On the other hand, the BDSDEL \eqref{penal} can be written as
\begin{equation}\label{penali}
Y^n_{t}=\xi_T +\int_{t}^{T}  f(s,Y^n_s,Z^n_s)ds+\int_{t}^{T} \widetilde{g}(s)dB_s+K^n_{T}-K^n_{t}-\sum_{k=1}^{\infty}\int_t^TZ_{s}^{(k),n}dH_{s}^{(k)}
\end{equation}
where
$$K^{n}_t:=K^{n,\ast}_t+\sum_{0\leq s<t}\Delta_+K^{n}_s=n\int_0^t(Y^n_{s}-\xi_{s})^-ds+\sum_{0\leq \sigma_{n,i}<t}(Y^n_{\sigma_{n,i}+}-\xi_{\sigma_{n,i}})^-.$$
It remains to establish the convergence of the sequence $(Y^n,Z^n,K^n)_{n\geq 1}$ to the solution of the RBDSDEL associated with parameters $(f,\widetilde{g},\xi)$. For this end, we divided the proof into four steps:
~\\
\textbf{Step 1: A priori estimate.}
~\\
There exists a positive constant $\mathcal{C}_{_{\beta,L}}$ independent on $n$ such that for all $\beta$ large enough
\begin{eqnarray*}
&&\mathbb{E}\operatorname*{ess
\,sup}_{\tau\in\mathcal{T}_{[0,T]}}e^{\beta \tau}|Y^n_{\tau}|^2+\mathbb{E}\int_0^Te^{\beta s}|Y^n_{s}|^2ds+\mathbb{E}\int_0^Te^{\beta s}\|Z^n_{s}\|_{\ell^2}^2ds+\mathbb{E}|K^n_T|^2\nonumber\\
 &\leq&\mathcal{C}_{_{\beta,L}}\left(\mathbb{E}\operatorname*{ess
\,sup}_{\tau\in\mathcal{T}_{[0,T]}}e^{2\beta \tau}|\xi_{\tau}|^2+\mathbb{E}\int_{0}^{T}e^{\beta s}|f(s,0,0)|^2ds+\mathbb{E}\int_0^Te^{\beta s}|\widetilde{g}(s)|^2ds\right).
\end{eqnarray*}

Indeed, by applying the Proposition \ref{pro} to $e^{\beta t}|Y^n_t|^2$, we have
\begin{eqnarray}\label{e2}
&&e^{\beta t}|Y^n_t|^2+\beta\int_t^Te^{\beta s}|Y^n_{s}|^2ds+\int_t^Te^{\beta s}\|Z^n_{s}\|_{\ell^2}^2ds\nonumber\\
 &=&e^{\beta T}|\xi_T|^2+2\int_{t}^{T}e^{\beta s}Y^n_{s}f(s,Y^n_s,Z^n_s)ds+2\int_{t}^{T}e^{\beta s}Y^n_{s}\widetilde{g}(s)dB_s\nonumber\\
 &&+2\int_{t}^{T}e^{\beta s}Y^n_{s-}dK^{n,\ast}_s-2\sum_{k=1}^{\infty}\int_t^Te^{\beta s}Y^n_{s-}Z_{s}^{(k),n}dH_{s}^{(k)}+\int_t^Te^{\beta s}|\widetilde{g}(s)|^2ds\nonumber\\
 &&-\sum_{t<s\leq T}e^{\beta s}|\Delta Y^n_{s}|^2-\sum_{t\leq s<T}e^{\beta s}(|Y^n_{s+}|^2-|Y^n_{s}|^2).
\end{eqnarray}
Observe that
\begin{equation*}
|Y^n_{s+}|^2-|Y^n_{s}|^2=|\Delta_+Y^n_s|^2+2Y^n_s\Delta_+Y^n_s=|\Delta_+Y^n_s|^2-2Y^n_s\Delta_+K^{n}_s.
\end{equation*}
Moreover, according to the Lipschitz property of $f$ we have
\begin{eqnarray*}
2\int_{t}^{T}e^{\beta s}Y^n_{s}f(s,Y^n_s,Z^n_s)ds&\leq&\varrho_{_2}\int_t^Te^{\beta s}|Y^n_{s}|^2ds+\frac{1}{\varrho_{_2}}\int_t^Te^{\beta s}\left|f(s,Y^n_s,Z^n_s)\right|^2ds\\
&\leq& \left(\varrho_{_2}+\frac{3L^2}{\varrho_{_2}}\right)\int_t^Te^{\beta s}|Y^n_{s}|^2ds+\frac{3L^2}{\varrho_{_2}}\int_t^Te^{\beta s}\|Z^n_{s}\|_{\ell^2}^2ds\\
&&+\frac{3}{\varrho_{_2}}\int_t^Te^{\beta s}\left|f(s,0,0)\right|^2ds
\end{eqnarray*}
for some $\varrho_{_2}>0$. On the other hand, for each $t\in(\sigma_{n,i-1},\sigma_{n,i}]$ it holds true that
\begin{eqnarray*}
\int_{0}^{t}e^{\beta s}Y^n_{s-}dK^{n,\ast}_s+\sum_{s<t}e^{\beta s}Y^n_s\Delta_+K^{n}_s&=&\int_{0}^{t}e^{\beta s}Y^n_{s-}dK^{n}_s\\
&=&\int_0^te^{\beta s}Y^n_{s}n(Y^n_s-\xi_s)^-ds\\
&=&\int_0^te^{\beta s}\xi_{s}n(Y^n_s-\xi_s)^-ds-\int_0^te^{\beta s}n((Y^n_s-\xi_s)^-)^2ds\\
&\leq&\int_0^te^{\beta s}\xi_{s}n(Y^n_s-\xi_s)^-ds=\int_{0}^{t}e^{\beta s}\xi_{s}dK^{n}_s\\
&\leq&\frac{\varrho_{_3}}{2}\operatorname*{ess
\,sup}_{\tau\in\mathcal{T}_{[0,T]}}e^{2\beta \tau}|\xi_{\tau}|^2+\frac{1}{2\varrho_{_3}}\left|K^{n}_T\right|^2,
\end{eqnarray*}
for some $\varrho_{_3}>0$. Plugging the above observations on \eqref{e2} and taking the expectation we obtain for $\varrho_{_2}>3L^2$ and $\beta>\varrho_{_2}+\frac{3L^2}{\varrho_{_2}}$
\begin{eqnarray*}
&&\mathbb{E}e^{\beta t}|Y^n_t|^2+\mathbb{E}\int_t^Te^{\beta s}|Y^n_{s}|^2ds+\mathbb{E}\int_t^Te^{\beta s}\|Z^n_{s}\|_{\ell^2}^2ds\nonumber\\
 &\leq&C_{_{\beta,L}}\left(\mathbb{E}\operatorname*{ess
\,sup}_{\tau\in\mathcal{T}_{[0,T]}}e^{2\beta \tau}|\xi_{\tau}|^2+\mathbb{E}\int_{t}^{T}e^{\beta s}|f(s,0,0)|^2ds+\mathbb{E}\int_t^Te^{\beta s}|\widetilde{g}(s)|^2ds\right.\nonumber\\
 &&\left.+\varrho_{_3}\mathbb{E}\operatorname*{ess
\,sup}_{\tau\in\mathcal{T}_{[0,T]}}e^{2\beta \tau}|\xi_{\tau}|^2+\frac{1}{\varrho_{_3}}\mathbb{E}\left|K^{n}_T\right|^2\right)
\end{eqnarray*}
where $C_{_{\beta,L}}$ is a positive constant which depends on $\beta$ and $L$. Moreover, since
$$K^n_T=Y^n_0-\xi_T-\int_0^Tf(s,Y^n_s,Z^n_s)ds-\int_{0}^{T} \widetilde{g}(s)dB_s+\sum_{k=1}^{\infty}\int_0^TZ_{s}^{(k),n}dH_{s}^{(k)},$$
then
\begin{eqnarray*}
\mathbb{E}|K^n_T|^2&\leq&5\left(\mathbb{E}|Y^n_0|^2+\mathbb{E}e^{\beta T}|\xi_T|^2+3L^2\mathbb{E}\int_0^Te^{\beta s}|Y^n_{s}|^2ds+(3L^2+1)\mathbb{E}\int_0^Te^{\beta s}\|Z^n_{s}\|_{\ell^2}^2ds\right.\nonumber\\
&&\left.+3\mathbb{E}\int_0^Te^{\beta s}\left|f(s,0,0)\right|^2ds
+\mathbb{E}\int_0^Te^{\beta s}\left|\widetilde{g}(s)\right|^2ds\right).
\end{eqnarray*}
Hence, for $\varrho_{_3}>5C_{_{\beta,L}}(3L^2+1)$
\begin{eqnarray*}
&&\mathbb{E}\int_0^Te^{\beta s}|Y^n_{s}|^2ds+\mathbb{E}\int_0^Te^{\beta s}\|Z^n_{s}\|_{\ell^2}^2ds+\mathbb{E}|K^n_T|^2\nonumber\\
 &\leq&C'_{_{\beta,L}}\left(\mathbb{E}\operatorname*{ess
\,sup}_{\tau\in\mathcal{T}_{[0,T]}}e^{2\beta \tau}|\xi_{\tau}|^2+\mathbb{E}\int_{0}^{T}e^{\beta s}|f(s,0,0)|^2ds+\mathbb{E}\int_0^Te^{\beta s}|\widetilde{g}(s)|^2ds\right).
\end{eqnarray*}
To conclude, both forward and backward version of Burkholder-Davis-Gundy inequality use in Pardoux and Peng \cite{PP2} yield
\begin{eqnarray*}
\mathbb{E}\operatorname*{ess
\,sup}_{\tau\in\mathcal{T}_{[0,T]}}e^{\beta \tau}|Y^n_{\tau}|^2
 &\leq&C''_{_{\beta,L}}\left(\mathbb{E}\operatorname*{ess
\,sup}_{\tau\in\mathcal{T}_{[0,T]}}e^{2\beta \tau}|\xi_{\tau}|^2+\mathbb{E}\int_{0}^{T}e^{\beta s}|f(s,0,0)|^2ds+\mathbb{E}\int_0^Te^{\beta s}|\widetilde{g}(s)|^2ds\right).
\end{eqnarray*}
Whence the desired result fellows.
~\\
\textbf{Step 2: There exists a process $Y$ with regulated trajectories such that $Y\geq \xi$ and $\mathbb{E}\left[\sup\limits_{0\leq t\leq T}|(Y^n_{t}-\xi_{t})^-|^2\right]\xrightarrow[n\to +\infty]{}0.$}
~\\
Recall that $Y^n$ satisfies the modified BDSDEL (\ref{mod}) with terminal value $\xi_{\sigma_{n,i}}\vee Y^n_{\sigma_{n,i}+}$ and first generator $f^n(.,y,z)=f(.,y,z)+n(y-\xi_{.})^-$ on each interval $(\sigma_{n,i-1},\sigma_{n,i}]$, $i=1,...,k_{n}+1$. Since $f^n(t,y,z)\leq f^{n+1}(t,y,z)$ then from the proposition \ref{comp} (see Appendix), we obtain that $Y^{n}\leq Y^{n+1}$ a.s. Hence there exists a process $Y$ such that $Y^n\nearrow Y$ a.s. Since $Y^n$ is bounded in $\mathcal{S}^2$, then by Fatou's lemma $Y$ is also bounded in $\mathcal{S}^2$.

Moreover, thanks to step 1, the sequences $(Z^n)_{n\geq0}$ and $(f(.,Y^n,Z^n))_{n\geq0}$ are bounded in $\mathcal{H}^2(\ell^2)$ and $\mathcal{H}^2$ respectively. Then we can extract
subsequences which weakly converge in the related space. We note $\mathcal{Z}$ and $\overline{f}$ the respective weak limits. Henceforth, for every stopping time $\tau\in\mathcal{T}_{[0,T]}$, the following weak convergence holds
$$\int_{0}^{\tau}Z_{s}^{(k),n}dH_{s}^{(k)}\rightharpoonup \int_{0}^{\tau}\mathcal{Z}^{(k)}_{s}dH_{s}^{(k)}\quad\mbox{and}\quad\int_0^\tau f(s,Y^n_s,Z^n_s)ds\rightharpoonup \int_0^\tau \overline{f}(s)ds$$
as $n\to +\infty$. Next, from the equation
\begin{equation*}
K^n_{\tau}=Y^n_0-Y^n_\tau -\int_0^\tau f(s,Y^n_s,Z^n_s)ds-\int_0^\tau \widetilde{g}(s)dB_s+\sum_{k=1}^{\infty}\int_0^\tau Z_{s}^{(k),n}dH_{s}^{(k)}.
\end{equation*}
we have the following weak convergence
\begin{equation*}
K^n_{\tau}\rightharpoonup \mathcal{K}_\tau:=Y_0-Y_\tau -\int_0^\tau \overline{f}(s)ds-\int_0^\tau \widetilde{g}(s)dB_s+\sum_{k=1}^{\infty}\int_0^\tau \mathcal{Z}^{(k)}_{s}dH_{s}^{(k)}.
\end{equation*}
From Fatou's lemma we have $\mathbb{E}|\mathcal{K}_T|^2\;\leq\;\liminf_{n\rightarrow +\infty}\mathbb{E}|K^n_T|^2<+\infty$.
Moreover, since the process $(K_t^n)_{t\leq T}$ is nondecreasing predictable process with $K_0^n=0$, then the weak limit process $(\mathcal{K}_t)_{t\leq T}$ is also nondecreasing predictable process with $\mathcal{K}_0=0$.

On the other hand, we know that $\int_\sigma^\tau(Y_{s}-Y^n_{s})dK^{n,\ast}_s+ \sum_{\sigma\leq s<\tau}(Y_{s}-Y^n_{s})\Delta_+K^{n}_{s}\geq 0$ for some stopping times $\sigma, \tau\in\mathcal{T}_{[0,T]}$ such that $\sigma\leq\tau$, then $\liminf_{n\to +\infty}\left(\int_\sigma^\tau(Y_{s}-Y^n_{s})dK^{n,\ast}_s+ \sum_{\sigma\leq s<\tau}(Y_{s}-Y^n_{s})\Delta_+K^{n}_{s}\right)\geq 0$. Also, It is easy to see that $\Delta_-K_t^n=0$ for $n\geq 0$ and $t\leq T$.
Therefore, thanks to the Marzougue's monotonic limit theorem (Theorem 2.1 in \cite{Ma2}), the processes $Y$ and $\mathcal{K}$ are regulated, and $\int_{0}^{t}\overline{f}(s)ds=\int_{0}^{t}f(s,Y_s,Z_s)ds$ for all $t\leq T$ a.s. where $Z$ is the strong limit of $(Z^n)_{n\geq0}$.

Next, according to the boundedness of $(K_t^n)_{t\leq T}$ we deduce that $\mathbb{E}\int_0^T(Y_{s}-\xi_{s})^-ds=0$ which implies that $Y_{t}\geq \xi_{t}$  for all $t\leq T$ $\mathbb{P}$-a.s. In particular, $(Y^n_t-L_t)^-\searrow 0$ for all $t\leq T$ $\mathbb{P}$-a.s.
Consequently, from a generalized Dini's lemma (see page 202 in \cite{DM2}), we have $\sup_{0\leq t\leq T}(Y^n_{t}-L_{t})^-\searrow 0$ for all $t\leq T$ $\mathbb{P}$-a.s. Therefore, since $|(Y^n_{t}-L_{t})^-|\leq |Y^0_t|+|L^+_t|$, the Lebesgue's dominated convergence theorem implies that
$$\mathbb{E}\left[\sup_{0\leq t\leq T}|(Y^n_t-L_t)^-|^2\right]\xrightarrow[n\to +\infty]{}0\quad \mbox{a.s}.$$
~\\
\textbf{Step 3: Strong convergence result.}
~\\
There exists an adapted process $(Y,Z,K)$ such that
\begin{eqnarray*}
 &&\|Y^{n}-Y\|_{\mathcal{S}^2}^{2}+\|Z^{n}-Z\|_{\mathcal{H}^2(\ell^2)}^{2}+\|K^{n}-K\|_{\mathcal{S}^2}^{2}\xrightarrow[n\to +\infty]{}0.
\end{eqnarray*}
Indeed, let us put $\Re^{n,p}=\Re^n-\Re^p$ for each $n\geq p\geq0$ and for $\Re\in\{Y,Z,K\}$. By applying the Proposition \ref{pro} to $|Y^{n,p}_t|^2$, we have
\begin{eqnarray*}
|Y^{n,p}_t|^2+\int_t^T\|Z^{n,p}_{s}\|_{\ell^2}^2ds
 &=&2\int_{t}^{T}Y^{n,p}_{s}(f(s,Y^n_s,Z^n_s)-f(s,Y^p_s,Z^p_s))ds+2\int_{t}^{T}Y^{n,p}_{s-}dK^{n,p,\ast}_s\nonumber\\
 &&-2\sum_{k=1}^{\infty}\int_t^TY^{n,p}_{s-}Z_{s}^{(k),n,p}dH_{s}^{(k)}-\sum_{t<s\leq T}|\Delta Y^{n,p}_{s}|^2-\sum_{t\leq s<T}(|Y^{n,p}_{s+}|^2-|Y^{n,p}_{s}|^2).
\end{eqnarray*}
Remark that
\begin{equation*}
|Y^{n,p}_{s+}|^2-|Y^{n,p}_{s}|^2=|\Delta_+Y^{n,p}_s|^2+2Y^{n,p}_s\Delta_+Y^{n,p}_s=|\Delta_+Y^{n,p}_s|^2-2Y^{n,p}_s\Delta_+K^{n,p}_s,
\end{equation*}
and
$$\int_{t}^{T}Y_{s-}^{n,p}dK_{s}^{n,p}\leq -\int_{t}^{T}(Y_{s-}^n-\xi_s)dK_{s}^{p}\leq\sup_{0\leq t\leq T}(Y_{t}^{n}-\xi_{t})^-K_{T}^{p}.$$
Moreover, according to the Lipschitz property of $f$ we have
\begin{eqnarray*}\label{e3}
2\int_{t}^{T}Y^{n,p}_{s}(f(s,Y^n_s,Z^n_s)-f(s,Y^p_s,Z^p_s))ds
&\leq& \left(\varrho_{_4}+\frac{2L^2}{\varrho_{_4}}\right)\int_t^T|Y^{n,p}_{s}|^2ds+\frac{2L^2}{\varrho_{_4}}\int_t^T\|Z^{n,p}_{s}\|_{\ell^2}^2ds
\end{eqnarray*}
for some $\varrho_{_4}>0$.
Consequently, by Gronwall's lemma and for $\varrho_{_4}>2L^2$ we get
\begin{eqnarray*}
&&\mathbb{E}|Y^{n,p}_t|^2+\mathbb{E}\int_t^T\|Z^{n,p}_{s}\|_{\ell^2}^2ds\leq \left(\mathbb{E}\left[\sup_{0\leq t\leq T}|(Y_{t}^{n}-\xi_{t})^-|^2\right]\right)^\frac{1}{2}.\left(\mathbb{E}\left[K_{T}^{p}|^2\right]\right)^\frac{1}{2}\xrightarrow[n\to +\infty]{}0.
\end{eqnarray*}
It follows that $(Z^n)_{n\geq0}$ is a Cauchy sequence in the complete space $\mathcal{H}^{2}(\ell^2)$. Then, there exists a process $Z\in\mathcal{H}^{2}(\ell^2)$ such that the sequence $(Z^n)_{n\geq0}$ converges toward $Z$. On the other hand, by applying Burkholder-Davis-Gundy's inequality, one can derive that
$\mathbb{E}\sup_{0\leq t\leq T}|Y_{t}^{n,p}|^{2}\to 0$ as $n,p\to +\infty$. Then
$$\mathbb{E}\left[\sup_{0\leq t\leq T}|Y_{t}^{n}-Y_{t}|^{2}\right]\xrightarrow[n\to +\infty]{}0$$
where $Y$ is regulated and belongs to $\mathcal{S}^2$.
Now, since
$$K^n_t=Y^n_0-Y^n_t-\int_0^tf(s,Y^n_s,Z^n_s)ds-\int_{0}^{t} \widetilde{g}(s)dB_s+\sum_{k=1}^{\infty}\int_0^tZ_{s}^{(k),n}dH_{s}^{(k)},$$
hence $\mathbb{E}\sup_{0\leq t\leq T}|K_{t}^{n,p}|^{2} \to 0$ as $n,p\to +\infty$.
It follows that $(K^n)_{n\geq0}$ is a Cauchy sequence in $\mathcal{S}^2$, then there exists an optional process $K$ limit uniform to $K^n$. Consequently,
$$\mathbb{E}\left[\sup_{0\leq t\leq T}|K_{t}^{n}-K_{t}|^{2}\right]\xrightarrow[n\to +\infty]{}0.$$

~\\
\textbf{Step 4: Conclusion: The limiting process $(Y,Z,K)$ solve the RBDSDEL.}
~\\
The limiting process $(Y,Z,K)$ is the solution of RBDSDEL associated with parameters $(f,\widetilde{g},\xi)$.
Indeed, from step 2, the regulated process $Y$ has the form
\begin{equation}\label{equa33}
Y_t=\xi_T+\int_t^T\overline{f}(s)ds+\int_{t}^{T}\widetilde{g}(s)dB_s+\mathcal{K}_T-\mathcal{K}_t-\sum_{k=1}^{\infty}\int_t^T\mathcal{Z}_{s}^{(k)}dH_{s}^{(k)}\quad \forall t\leq T.
\end{equation}
On the other hand, we have
\begin{eqnarray*}
\mathbb{E}\int_0^T|f(s,Y_s,Z_s)-f(s,Y^n_s,Z^n_{s})|^2ds
&\leq&2L\left(T\mathbb{E}\sup_{0\leq t\leq T}|Y_{t}-Y^n_t|^2+\mathbb{E}\int_0^T\|Z_{s}-Z^n_s\|_{\ell^2}^2ds\right)\xrightarrow[n\to +\infty]{}0
\end{eqnarray*}
thanks to step 3. Then by passing to the limit as $n\to+\infty$ in (\ref{penali}) we get
\begin{equation}\label{equa3}
Y_t=\xi_T+\int_t^Tf(s,Y_s,Z_s)ds+\int_{t}^{T}\widetilde{g}(s)dB_s+K_T-K_t-\sum_{k=1}^{\infty}\int_t^TZ_{s}^{(k)}dH_{s}^{(k)}\quad \forall t\leq T.
\end{equation}
Then by comparing the forward form of (\ref{equa33}) and (\ref{equa3}), we obtain
\begin{equation*}
K_t-\mathcal{K}_t=\sum_{k=1}^{\infty}\int_0^t(Z_{s}^{(k)}-\mathcal{Z}_{s}^{(k)})dH_{s}^{(k)}.
\end{equation*}
Since every predictable martingale of finite variation is constant, we conclude that $\mathcal{K}\equiv K$ and $\mathcal{Z}\equiv Z$.
Now, let us prove that $Y_t\geq \xi_t$ for all $t\leq T$. From step 3, up to a subsequence, $(Y_t^n-\xi_t)^-$ tends to zero $\mathbb{P}$-a.s. for a dense subset of $t$. Hence $Y_t\geq \xi_t$ for a dense subset of $t$. Consequently, $Y_{t+}\geq \xi_{t+}$ for each $t\in[0,T)$. In fact, $Y_{t}\geq \xi_{t}$ for each $t\in[0,T)$. Indeed, if
$\Delta_+\xi_t\geq0$ for some $t\in[0,T)$ then $Y_t=-\Delta_+Y_t+Y_{t+}\geq Y_{t+}\geq \xi_{t+}\geq \xi_{t}$ whereas if $\Delta_+\xi_t<0$ for some $t\in[0,T)$ then $t\in\bigcup_{i}[[\sigma_{n,i}]]$ for sufficiently large $n$, which implies that $\Delta_+K_t^n=(Y^n_{t+}-\xi_t)^-$. Suppose that $Y_t^n\leq \xi_t$ for some $t$. Since $\Delta_+Y_t^n=-\Delta_+K_t^n$, thus we have
$Y^n_{t+}-\xi_t<Y^n_{t+}-Y_t^n=-(Y^n_{t+}-\xi_t)^-$, which leads to a contradiction. Thus $Y^n_t\geq\xi_t$ for each $t\in[0,T)$, and hence $Y_t\geq\xi_t$ for each $t\in[0,T)$. Consequently, $Y_t\geq\xi_t$ for each $t\in[0,T]$.
It remains to show the Skorokhod condition for the regulated process $K$. Since $Y_{t}+\int_{0}^{t} f(s,Y_{s},Z_{s})ds+\int_{0}^{t}\widetilde{g}(s)dB_s$ is a supermartingale of class (D) and by using the convergence result of the sequence $(Y^n,Z^n,K^n)_{n\geq 0}$ we have
$$Y_t=\operatorname*{ess
\,sup}\limits_{\tau\in\mathcal{T}_{[t,T]}}\mathbb{E}\left[\xi_\tau+\int_t^\tau f(s,Y_{s},Z_{s})ds+\int_{0}^{t}\widetilde{g}(s)dB_s|\mathcal{G}_t\right].$$
Denote
$$\eta_t=\xi_t+\int_0^t f(s,Y_{s},Z_{s})ds+\int_{0}^{t}\widetilde{g}(s)dB_s-\mathbb{E}\left[\xi_T+\int_0^T f(s,Y_{s},Z_{s})ds+\int_{0}^{T}\widetilde{g}(s)dB_s|\mathcal{G}_t\right]$$
which is regulated process with $\eta_T=0$ and $\sup_{0\leq t\leq T}|\eta_t|\in\mathrm{L}^2(\Omega)$. Let $\mathbf{Sn}(\eta)$ be its Snell envelope. We have $\mathbf{Sn}(\eta)\in\mathcal{S}^2$ and then $(\mathbf{Sn}(\eta_t))_{t\leq T}$ is of class (D). Henceforth, from Mertens decomposition (Theorem \ref{decom} in Appendix) there exists a unique regulated increasing process $\widetilde{K}$ and a unique local martingale $\widetilde{M}$ such that
$$\mathbf{Sn}(\eta_t)=Y_t-\mathbb{E}\left[\xi_T+\int_t^T f(s,Y_{s},Z_{s})ds+\int_{t}^{T}\widetilde{g}(s)dB_s|\mathcal{G}_t\right]=\widetilde{M}_t-\widetilde{K}_t.$$
Furthermore, by applying the predictable representation (Proposition \ref{prp} in Appendix) to the martingale
$$\mathbb{E}\left[\xi_T+\int_0^T f(s,Y_{s},Z_{s})ds+\int_{0}^{T}\widetilde{g}(s)dB_s|\mathcal{G}_t\right]+\widetilde{M}_t,$$
there exists a unique predictable process $\widetilde{Z}$ such that
$$Y_{t}=\xi_T +\int_{t}^{T} f(s,Y_{s},Z_{s})ds+\int_{t}^{T}\widetilde{g}(s)dB_s+\widetilde{K}_{T}-\widetilde{K}_{t}-\sum_{k=1}^{\infty}\int_t^T\widetilde{Z}_{s}^{(k)}dH_{s}^{(k)}.$$
Thanks to the uniqueness of the solution to RBDSDEL, $\widetilde{Z}\equiv Z$ and $\widetilde{K}\equiv K$. Finally, from Corollary 3.11 in \cite{KRS} we get
\begin{eqnarray*}
\int_0^T(Y_{t-}-\hat{\xi}_{t})dK^{\ast}_t+\sum_{t<T}(Y_{t}-\xi_{t})\Delta_+K_t
&=&\displaystyle\int_0^T(\mathbf{Sn}(\eta_{t-})-\hat{\eta}_{t})dK^{\ast}_t
+\sum_{t<T}(\mathbf{Sn}(\eta_{t})-\eta_{t})\Delta_+K_t=0 \quad \mbox{a.s.}
\end{eqnarray*}
Whence the proof is complete.
\end{proof}

The main result of this section is the following:

\begin{theorem}
Let $(f,g)$ be a pair of Lipschitz drivers and $\xi$ be a completely irregular barrier. Then the RBDSDEL associated with parameters $(f,g,\xi)$ admits a unique solution.
\end{theorem}
\begin{proof}
Denote $\mathfrak{B}^2_\beta:=\mathcal{H}^{2}_\beta\times\mathcal{H}^{2}_\beta(\ell^2)$ the Banach space endowed with the norm
$$\|(Y,Z)\|^2_{\mathfrak{B}^{2}_\beta}=\mathbb{E}\int_{0}^{T}e^{\beta t}|Y_{t}|^2dt+\mathbb{E}\int_{0}^{T}e^{\beta t}\|Z_{t}\|_{\ell^2}^2dt.$$
Given $(y,z)\in \mathfrak{B}^2_\beta$ and consider the following RBDSDEL
\begin{equation}\label{bsdefp}
Y_{t}=\xi +\int_{t}^{T} f(s,Y_s,Z_s)ds+\int_{t}^{T}g(s,y_s,z_s)dB_{s}+K_T-K_t-\sum_{k=1}^{\infty}\int_t^TZ_{s}^{(k)}dH_{s}^{(k)}.
\end{equation}
By the Lipschitz property on $g$, we have $\mathbb{E}\int_0^T|g(t,y_t,z_t)|^2dt<+\infty$. Then, from Lemma \ref{tt}, the RBDSDEL (\ref{bsdefp}) admits a unique solution.
Next, we define a mapping $\Phi$ from $\mathfrak{B}^2_\beta$ into itself such that for any $(y,z)$ and $(y',z')$ in $\mathfrak{B}^2_\beta$, $\Phi(y,z)=(Y,Z)$ and $\Phi(y',z')=(Y',Z')$ where $(Y,Z,K)$ and $(Y',Z',K')$ are the solutions of the RBDSDEL associated with parameters $(f,g(.,y,z),\xi)$ and $(f,g(.,y',z'),\xi)$ respectively. Set $\bar{\Re}=\Re-\Re'$ for $\Re\in\{Y,Z,K,y,z\}$, and we put $\bar f_t=f(t,Y_t,Z_t)-f(t,Y'_t,Z'_t)$ and $\bar g_t=g(t,y_t,z_t)-g(t,y'_t,z'_t)$ for all $t\leq T$. By applying the Proposition \ref{pro} to $e^{\beta t}|\bar Y_{t}|^2$ we have
\begin{eqnarray*}
&&\mathbb{E}e^{\beta t}|\bar Y_{t}|^2+\beta\mathbb{E}\int_t^Te^{\beta s}|\bar Y_{s}|^2ds+\mathbb{E}\int_t^Te^{\beta s}\|\bar Z_{s}\|_{\ell^2}^2ds\nonumber\\
 &=&2\mathbb{E}\int_{t}^{T}e^{\beta s}\bar Y_{s}\bar f_s ds
 +2\mathbb{E}\int_{t}^{T}e^{\beta s}\bar Y_{s-}d\bar K^{\ast}_s+\mathbb{E}\int_t^Te^{\beta s}|\bar g_s|^2ds\nonumber\\
 &&-\mathbb{E}\sum_{t<s\leq T}e^{\beta s}|\Delta \bar Y_{s}|^2-\mathbb{E}\sum_{t\leq s<T}e^{\beta s}|\Delta_+ \bar Y_{s}|^2
 +2\mathbb{E}\sum_{t\leq s<T}e^{\beta s}\bar Y_{s}\Delta_+\bar K_s.
\end{eqnarray*}
Thanks to the Skorokhod condition on $K$, we have $\int_{t}^{T} e^{\beta s}\bar Y_{s-}d\bar K^{\ast}_s\leq 0$
and
\begin{eqnarray*}
\sum_{t\leq s<T}e^{\beta s}\bar Y_{s}\Delta_+\bar K_s
&=&\sum_{t\leq s<T}e^{\beta s}(Y_{s}-\xi_{s})\Delta_+K_{s}-\sum_{t\leq s<T}e^{\beta s}(Y'_{s}-\xi_{s})\Delta_+K_{s}\\
&&+\sum_{t\leq s<T}e^{\beta s}(Y'_{s}-\xi_{s})\Delta_+K'_{s}-\sum_{t\leq s<T}e^{\beta s}(Y_{s}-\xi_{s})\Delta_+K'_{s}\\
&\leq& 0.
\end{eqnarray*}
Moreover, from the Lipschitz property on $f$ and $g$ we deduce for some $\varrho_{_4}>0$
\begin{eqnarray*}
2\int_{t}^{T}e^{\beta s}\bar Y_{s}\bar f_s ds
&\leq&2\int_{t}^{T}e^{\beta s}L\bar Y_{s}(\bar Y_s+\|\bar Z_s\|_{\ell^2})ds\\
&\leq&2L(1+L)\int_{t}^{T}e^{\beta s}|\bar Y_{s}|^2ds+\frac{1}{2}\int_{t}^{T}e^{\beta s}\|\bar Z_{s}\|_{\ell^2}^2ds
\end{eqnarray*}
and
\begin{equation*}
\int_t^Te^{\beta s}|\bar g_s|^2ds\leq L\int_{t}^{T}e^{\beta s}|\bar y_{s}|^2ds+\alpha\int_{t}^{T}e^{\beta s}\|\bar z_{s}\|_{\ell^2}^2ds.
\end{equation*}
Then
\begin{eqnarray*}
\mathbb{E}\left(\Lambda_1\int_t^Te^{\beta s}|\bar Y_{s}|^2ds+\frac{1}{2}\int_t^Te^{\beta s}\|\bar Z_{s}\|_{\ell^2}^2ds\right)
 &\leq&2\alpha\mathbb{E}\left(\Lambda_1\int_t^Te^{\beta s}|\bar y_{s}|^2ds+\frac{1}{2}\int_t^Te^{\beta s}\|\bar z_{s}\|_{\ell^2}^2ds\right)
\end{eqnarray*}
where $\Lambda_1=\frac{L}{2\alpha}$ and $\beta$ is chosen such that $\beta>\Lambda_1+2L(1+L)$. Consequently, the mapping $\Phi$ is a contraction and then has a unique fixed point $(Y,Z)$ which actually belongs to $\mathfrak{B}^2_\beta$. Moreover, there exists $K\in\mathcal{S}^2$ $(K_0=0)$ such that $(Y,Z,K)$ is the unique solution of the RBDSDEL associated with parameters $(f,g,\xi)$.
\end{proof}

\section{Appendix}
In this section we summarize the principal tools used in our proofs throughout the paper.

\begin{definition}\label{left}
Let $\tau\in\mathcal{T}_{[0,T]}$. An optional process $(\xi_t)_{t\leq T}$ is said to be left upper-semicontinuous along stopping times at the stopping time $\tau$ if for all nondecreasing sequence of stopping times $(\tau_n)_{n\geq 0}$ such that $\tau_n\uparrow\tau$ a.s., $\xi_\tau\geq\limsup_{n\rightarrow+\infty}\xi_{\tau_n}$ a.s. The process $(\xi_t)_{t\leq T}$ is said to be left upper-semicontinuous along stopping times if it is left upper-semicontinuous along stopping times at each $\tau\in\mathcal{T}_{[0,T]}$.
\end{definition}

\begin{proposition}[Predictable representation property of L\'{e}vy processes (Nualart and Schoutens \cite{NS:2000})]\label{prp}
Every random variable $M$ in $\mathbb{L}^2(\Omega,\mathcal{G})$ has a representation of the form
$$M=\mathbb{E}[M]+\sum_{k=1}^{\infty}\int_0^TZ_{s}^{(k)}dH_{s}^{(k)}$$
where $\left\{Z_{s}^{(k)};\;\;k=1,...\infty\right\}$ are predictable.
\end{proposition}

\begin{theorem}[Mertens decomposition (cf. Theorem 20 page 429 in \cite{DM2} or page 528 in \cite{Len:1980})]\label{decom}
 Let $\tilde{Y}$ be a strong optional supermartingale of class(D). There exists a unique uniformly integrable locale martingale $M$ and a unique nondecreasing predictable process $K$ (not necessarily right or left continuous) with $K_0=0$ and $\mathbb{E}[K_T]<+\infty$ such that
\begin{equation*}
\tilde{Y}_\tau=M_{\tau}-K_\tau\qquad \forall\tau\in\mathcal{T}_{[0,T]} \; a.s.
\end{equation*}
\end{theorem}

\begin{proposition}[It\^{o}'s formula for regulated processes]\label{pro}
Let $Y$ be a semimartingale with regulated trajectories and $F$ be a twice continuously differentiable function on $\mathbb{R}^n$. Then, almost surely, for each $n\in\mathbb{N}$ and all $t\geq0$,
\begin{eqnarray*}
F(Y_{t})
&=&F(Y_{0})
+\sum_{k=1}^n\int_0^tD^k F(Y_{s-})dY^{\ast,k}_s+\frac{1}{2}\sum_{k,l=1}^n\int_0^tD^kD^l F(Y_{s-})d[Y^{\ast,k},Y^{\ast,l}]^{c}_s\\
&&+\sum_{0<s\leq t}\left[F(Y_s)-F(Y_{s-})-\sum_{k=1}^nD^k F(Y_{s-})\Delta Y^k_s\right]+\sum_{0\leq s<t}\left[F(Y_{s+})-F(Y_{s})\right],
\end{eqnarray*}
where $D^k$ denotes the differentiation operator with respect to the $k$-th coordinate, and $[.,.]^{c}$ denotes the continuous part of the quadratic variation of corresponding process.
\end{proposition}

In what follows a special comparison theorem for the solutions to BDSDEs without reflection.
\begin{proposition}[BDSDE's comparison theorem]\label{comp}
Let $(Y^{i},Z^{i})$ be a solution of the following BDSDE (associated with parameters $(\xi^i,f^i,g)$)
$$Y^i_{t}=\xi^i +\int_{t}^{T} f^i(s,Y^i_{s},Z^i_{s})ds+\int_{t}^{T} \widetilde{g}(s)dB_s-\sum_{k=1}^{\infty}\int_t^TZ_{s}^{(k),i}dH_{s}^{(k)}$$
where $f^i$ are Lipschitz drivers for $i\in\{1,2\}$, $\mathbb{E}\int_0^T|\widetilde{g}(t)|^2dt<+\infty$ and $\mathbb{E}|\xi|^2<+\infty$.
We suppose that $\xi^1\leq\xi^2$, $f^1(t,y,z) \leq f^2(t,y,z)$ $\forall(t,y,z)\in [0,T]\times\mathbb{R}\times\ell^2$ and
$$ \zeta^k_t=\frac{f^1\left(t,Y^2_t,\tilde Z^{(k-1)}_t\right)-f^1\left(t,Y^2_t,\tilde Z^{(k)}_t\right)}{Z^{(k),1}_t-Z^{(k),2}_t}\mathds{1}_{\{Z^{(k),1}_t-Z^{(k),2}_t\neq 0\}},$$
where
\begin{equation*}
\tilde Z^{(k)}_t=\left(Z^{(1),2}_t,Z^{(2),2}_t,\ldots,Z^{(k),2}_t,Z^{(k+1),1}_t,\ldots,Z^{(d),1}_t\right)
\end{equation*}
 such that
\begin{equation}\label{assum2}
\sum_{k=1}^{\infty}\zeta_{t}^{k}\Delta H_{t}^{(k)}>-1\qquad dt\otimes d\mathbb{P}-\mbox{a.s.}
\end{equation}
 Then $\forall t \leq T$, $Y^1_t \leq Y^2_t$ a.s.
\end{proposition}
\begin{proof}
Define $\widehat{\Re}=\Re^{1}-\Re^{2}$ for $\Re\in\{Y,Z,\xi\}$. Then the process $(\widehat{Y},\widehat{Z})$ satisfies the following equation

\begin{equation*}
\widehat{Y}_t =\widehat{\xi}+ \int_t^T\left(p_s\widehat{Y}_s+\sum_{k=1}^{\infty}\zeta^k_s\widehat{Z}_{s}^{(k)}+u_s\right)ds
-\sum_{k=1}^{\infty}\int_t^T\widehat{Z}_{s}^{(k)}dH_{s}^{(k)}
\end{equation*}
 where
\begin{itemize}
  \item $ p_t=\displaystyle\frac{f^1(t,Y^1_t,Z^{1}_t)-f^1(t,Y^2_t,Z^{1}_t)}{Y^{1}_t-Y^{2}_t}\mathds{1}_{\{Y^{1}_t-Y^{2}_t\neq 0\}}$;
 \item $u_t=f^1(t,Y^2_t,Z^{2}_t)-f^2(t,Y^2_t,Z^{2}_t)$.
\end{itemize}
Since the solutions of BDSDE are square integrable then thanks to Theorem 37 page 84 in Protter \cite{Pro}, for $0\leq s\leq t\leq T$, the following linear SDE
$$\Gamma_{s,t}=1+\int_s^t\Gamma_{s,r-}d\mathcal{X}_r$$
with
$$\mathcal{X}_t=\int_0^tp_sds+\sum_{k=1}^{\infty}\int_0^t\zeta_{s}^{k}dH_{s}^{(k)}$$
admits a unique solution of the form
\begin{equation*}
\Gamma_{s,t} = \exp\left(\mathcal{X}_t-\mathcal{X}_s\right)\times\prod_{s\leq r\leq t}\left(1+\sum_{k=1}^{\infty}\zeta_{r}^{k}\Delta H_{r}^{(k)}\right)\exp\left(-\sum_{k=1}^{\infty}\zeta_{r}^{k}\Delta H_{r}^{(k)}\right).
\end{equation*}
This solution is strictly positive according to the assumption \eqref{assum2}. Now, by applying It\^{o}'s formula to $\Gamma_{s,t}\widehat{Y}_t$ we get
\begin{eqnarray}\label{equacom}
\Gamma_{s,t}\widehat{Y}_t
&=&\Gamma_{s,T}\widehat{\xi}-\int_t^T\Gamma_{s,r-}d\widehat{Y}_r-\int_t^T\widehat{Y}_rd\Gamma_{s,r}-\int_t^Td[\Gamma,\widehat{Y}]_r\nonumber\\
&=&\Gamma_{s,T}\widehat{\xi}+\int_t^T\Gamma_{s,r}\left(p_r\widehat{Y}_r+\sum_{k=1}^{\infty}\zeta^k_r\widehat{Z}_{r}^{(k)}+u_r\right)dr
-\sum_{k=1}^{\infty}\int_t^T\Gamma_{s,r-}\widehat{Z}_{r}^{(k)}dH_{r}^{(k)}\\
&&-\int_t^T\Gamma_{s,r}p_r\widehat{Y}_rdr
-\sum_{k=1}^{\infty}\int_t^T\Gamma_{s,r-}\widehat{Y}_r\zeta^k_rdH_{r}^{(k)}
-\sum_{k=1}^{\infty}\int_t^T\Gamma_{s,r}\zeta^k_r\widehat{Z}_{r}^{(k)}d[H^{(k)},H^{(k)}]_r.\nonumber
\end{eqnarray}
But,
\begin{eqnarray*}
\mathbb{E}\left(\sum_{k=1}^{\infty}\int_t^T\Gamma_{s,r}\zeta^k_r\widehat{Z}_{r}^{(k)}d[H^{(k)},H^{(k)}]_r|\mathcal{G}_t\right)
&=&\mathbb{E}\left(\sum_{k=1}^{\infty}\int_t^T\Gamma_{s,r}\zeta^k_r\widehat{Z}_{r}^{(k)}d\langle H^{(k)},H^{(k)}\rangle_r|\mathcal{G}_t\right)\\
&=&\mathbb{E}\left(\sum_{k=1}^{\infty}\int_t^T\Gamma_{s,r}\zeta^k_r\widehat{Z}_{r}^{(k)}dr|\mathcal{G}_t\right).
\end{eqnarray*}
Thus, by taking the conditional expectation w.r.t $\mathcal{G}_t$ on both sides of the equality \eqref{equacom} we obtain
\begin{equation*}
\Gamma_{s,t}\widehat{Y}_t=\mathbb{E}\left(\Gamma_{s,T}\widehat{\xi}+\int_t^T\Gamma_{s,r}u_rdr|\mathcal{G}_t\right)\leq0
\end{equation*}
in view of $\Gamma_{s,r}>0$, $\widehat{\xi}\leq 0$ and $u_r\leq 0$. Consequently, $\forall t\leq T$  $Y_t^1 \leq Y_t^2$ a.s.
\end{proof}

\end{document}